\documentclass[10pt,a4paper]{article}
\usepackage[utf8]{inputenc}
\usepackage[english]{babel}
\usepackage{amsthm}
\usepackage{amsmath}
\usepackage{amsfonts}
\usepackage{amssymb}
\usepackage{tabu}
\newtheorem{theorem}{Theorem}[section]

\newtheorem{lem}[theorem]{Lemma}
\newtheorem{defn}[theorem]{Definition}
\newtheorem{conj}[theorem]{Conjecture}
\newtheorem{rem}[theorem]{Remark}
\newtheorem{pro}[theorem]{Proposition}
\author{Taboka Prince Chalebgwa}
\title{Sendov's Conjecture: A note on a paper of D\'{e}got}
\date{}
\begin{document}
\maketitle

\begin{abstract}
\noindent Sendov's conjecture states that if all the zeroes of a complex polynomial $P(z)$ of degree at least two lie in the unit disk, then within a unit distance of each zero lies a critical point of $P(z)$. In a paper that appeared in 2014, D\'{e}got proved that, for each $a\in (0,1)$, there exists an integer $N$ such that for any polynomial $P(z)$ with degree greater than $N$, if $P(a) = 0$ and all zeroes lie inside the unit disk, the disk $|z-a|\leq 1$ contains a critical point of $P(z)$. Based on this result, we derive an explicit formula $\mathcal{N}(a)$ for each $a \in (0,1)$ and, consequently obtain a uniform bound $N$ for all $a\in [\alpha , \beta]$ where $0<\alpha < \beta < 1$. This (partially) addresses the questions posed in D\'{e}got's paper.
\end{abstract}

\section{Introduction}

\noindent In this paper we are going to prove the following result:\\

\begin{theorem}
Let $a \in (0,1)$ and define $\mathcal{N}(a)$ to be

\begin{equation*}
\mathcal{N}(a) = \frac{20800}{a^7(1-a)^4}.
\end{equation*}

\noindent For any polynomial $P(z) = (z-a)\prod_{j=1}^{n-1}(z-z_j)$ with $n\geq \mathcal{N}(a)$ and $|z_j|\leq 1$ for all $j = 1, \ldots , n-1$, the disk $|z-a|\leq 1$ contains a critical point of $P(z)$.\\

\end{theorem}

\noindent The Gauss-Lucas theorem tells us that the critical points of a polynomial $P(z) \in \mathbb{C}[z]$ lie in the convex hull of its zeroes. The conjecture of Sendov, which seeks to obtain a more precise location of the critical points, is the following:\\

\begin{conj}{(Sendov, \cite{Hayman}, p.~25)}:
Let $P(z) = \prod_{j=1}^n(z-z_j)$ be a polynomial of degree $n \geq 2$ such that $|z_j|\leq 1$ for all $j = 1, \ldots , n$. Then each of the disks $|z-z_j|\leq 1$, $j=1, \ldots , n$ contains a critical point of $P(z)$.\\
\end{conj}

\noindent Over the years since its inception in 1958, many special cases of the conjecture have been established. For extensive surveys of these, the reader is referred to the books \cite{Rahman} and \cite{Small}. To give the ``modern" formulation of the conjecture though, we need a corollary of the following special case by Bojanov et-al:\\

\noindent \begin{lem} (\cite{Bojanov}, 1985): Let $P(z) = \prod_{j=1}^{n}(z-z_j)$, with $|z_j|\leq 1$ for $j=1, \ldots , n$. For each $j=1, \ldots ,n$ the closed disk

\begin{equation*}
|z-z_j| \leq (1+ |z_1 \cdots z_n|)^\frac{1}{n}
\end{equation*}

\noindent contains a critical point of $P(z)$.\\
\end{lem}

\noindent An immediate corollary of the above lemma is that Sendov's conjecture is true for polynomials of the form $P(z) = z \prod_{j=1}^{n-1}(z-z_j)$. In 1968, Rubinstein in \cite{Rubinstein} showed that Sendov's conjecture is also true ``at the zero $z_j$" of the polynomial $P(z)$ if $|z_j| = 1$. The above two special cases, together with the observation that a rotation of the plane preserves the relative configurations of the zeroes and critical points of the polynomial $P(z)$, means it is enough to consider the following reformulation of the conjecture:\\

\begin{conj}{(Sendov)}: Let
\begin{equation}
P(z) = (z-a)\prod_{j=1}^{n-1}(z-z_j), \ \text{with} \ a \in (0, 1), \ |z_j|\leq 1 \ \text{for} \ j=1, \ldots , n-1.
\end{equation}
\noindent Then the disk $|z-a|\leq 1$ contains a critical point of $P(z)$.\\
\end{conj}

\noindent \begin{rem} With the above reformulation in mind, we can henceforth talk of Sendov's conjecture being true (or false) at a particular zero $a \in (0,1)$ of the polynomial $P(z)$.
\end{rem}

\subsection{An overview of D\'{e}got's strategy}

\noindent Our paper is based on \cite{Degot}, a 2014 paper by Jerome D\'{e}got. In the paper, D\'{e}got proved that, for each $a\in (0,1)$, there exists an integer $N$ such that for any polynomial $P(z)$ with degree greater than $N$, if $P(a) = 0$ and all zeroes lie inside the unit disk, the disk $|z-a|\leq 1$ contains a critical point of $P(z)$. For the reader's convenience, below we give a brief and non-technical overview of the approach in \cite{Degot}, as well as an outline of our paper.\\

\noindent D\'{e}got starts by fixing a polynomial  $P(z)$ with a zero at $a \in (0,1)$ and degree $n$, assumed to contradict Sendov's conjecture at $a$ (that is, all the critical points of $P(z)$ are more than a unit distance away from $a$). By studying closely the geometry of $P(z)$, he obtains lower and upper bounds of the form $|P(c)| \leq 1+a$ and $|P(c)| \geq CK^n$ respectively, where $c\in (0,a)$, $C$ and $K$ are some specifically defined parameters.\\

\noindent He then proceeds to show that if $n$ is greater than some integer bound $N$ (we shall refer to this bound as $N(a)$ henceforth), a contradiction on the size of $|P(c)|$ ensues, and hence the disk $|z-a| \leq 1$ must have a zero of $P'(z)$.\\

\noindent Worthy of note is that aside from an existence proof, there was no explicit formula given for calculating $N(a)$ for any given $a \in (0,1)$. In fact, upon closer inspection, one notices that the method used to obtain it depended on additional parameters associated with the polynomial $P(z)$. More precisely, a crucial technical inequality that $N(a)$ has to satisfy depended on the quantity $m$, defined as the real part of the mean of the zeroes of $P(z)$. D\'{e}got does indicate afterwards that this dependence can be removed by using a certain estimate on the size of $m$. This leaves much work to be done to actually give an explicit bound, which we do here. Finally, D\'{e}got ends his paper by outlining a series of steps through which one can calculate the requisite degree bounds for some values of $a \in (0,1)$. This algorithmic procedure however does not indicate any obvious way of constructing an explicit formula.\\

\noindent Carefully following the treatment in D\'{e}got's paper, we extract information from and modify his \textbf{Theorems 5, 6} and \textbf{7}. Each of these theorems introduced conditions which, for a given $a \in (0,1)$, an integer bound $N_1$ ($N_2$ and $N_3$ respectively) has to satisfy in order to draw the requisite conclusions on the size of $|P(c)|$. By studying closely these conditions, we systematically remove the extra dependencies on other parameters, and obtain explicit and continuous analogues of the bounds $N_1, N_2$ and $N_3$. We shall refer to these new formulas as $\mathcal{N}_1(a), \mathcal{N}_2(a)$ and $\mathcal{N}_3(a)$ respectively.\\

\noindent This allows us to obtain the conclusions of each of D\'{e}got's main theorems and hence, ultimately his main result, but now with explicit constants which depend continuously on $a$. So we can then obtain, as a by-product of the continuity of our functions, a uniform bound $N$ independent of $a \in [\alpha, \beta]$ for any $0< \alpha < \beta < 1$. At the end of his paper, D\'{e}got asked if it is possible to obtain a uniform bound $N$ which works for all $a \in (0,1)$, or at least an explicit formula $N(a)$ which produces a large enough degree for any given $a$.\\

\noindent In the interest of ultimately obtaining a ``simple" explicit formula $\mathcal{N}(a)$, we will often in intermediate steps of our arguments replace complicated formulae with simpler estimates. This of course comes at the expense of sharpness.\\

\noindent The results in this paper came from investigations carried out in the author's masters thesis.\\

\section{On the paper of D\'{e}got}

\noindent We begin this section with a result that allows us to bound the arithmetic mean of the roots of a polynomial assumed to contradict Sendov's conjecture at $a \in (0,1)$.\\

\subsection{The mean of a polynomial assumed to contradict Sendov's conjecture}

\noindent \begin{defn}\label{v1} By the real part of the mean of the roots (which also coincides with that of the critical points) of a polynomial $P(z)$ of degree $n$, we are referring to the quantity 

\begin{equation*}
m = \frac{1}{n}\Re \left(\sum_{j=1}^{n}z_j  \right).
\end{equation*}
\end{defn}

\begin{lem}\label{mm1} (\cite{Degot}, Cor 1): Let $P(z)$ be a polynomial of degree $n$ assumed to contradict Sendov's conjecture at $a \in (0,1)$. Then:

\begin{equation}
m \leq \underset{\delta \in (0,a)}{\inf} \left( \frac{\delta}{2} - \frac{1}{\delta n} \log (1-\sqrt{1+\delta^2 -\delta a}) \right) .
\end{equation}
\end{lem}
\noindent From the above lemma, we come up with a formula $\mathcal{N}_0(a)$ such that whenever $n \geq \mathcal{N}_0(a)$, then $m$ is less than some explicit function of $a$. For our purposes, $m \leq \frac{a}{4}$ will suffice.\\

\noindent Upon gaining this control over the size of $m$, we can then remove the dependence on $m$ from other future parameters. This shall become clear when we call upon this new quantity later. In the meantime, let us extract the formula.\\

\noindent Consider the inequality

\begin{equation*}
\left( \frac{\delta}{2} - \frac{1}{\delta n} \log (1-\sqrt{1+\delta^2 -\delta a}) \right) \leq \frac{a}{4}, \ \ \text{for some} \ \ \delta \in (0, \frac{a}{4}].
\end{equation*}

\noindent This holds when:

\begin{equation*}
n \geq \frac{-4 \log (1-\sqrt{1+\delta^2 -\delta a}) }{a \delta -2\delta^2}.
\end{equation*}

\noindent Define $N_0 (a)$ to be:

\begin{equation*}
N_0(a) = \left. \frac{-4 \log (1-\sqrt{1+\delta^2 -\delta a}) }{a \delta -2\delta^2} \right|_{\delta = \frac{a}{4}}.
\end{equation*}

\noindent This simplifies to:

\begin{equation*}
N_0 (a) = \frac{32 \log \left( \frac{4}{4 - \sqrt{16-3a^2}} \right)}{a^2}.
\end{equation*}

\noindent Furthermore, we note that $4 - \sqrt{16-3a^2} > \frac{a^2}{10}$ for all $a \in (0,1)$. We then, for the sake of simplicity define $\mathcal{N}_0 (a)$ to be:

\begin{equation*}
\mathcal{N}_0 (a) = \frac{32 \log \left( \frac{40}{a^2} \right)}{a^2}.
\end{equation*}

\noindent By Lemma \ref{mm1}, we can conclude that for any polynomial $P$ assumed to contradict Sendov's conjecture at $a \in (0,1)$, if $\deg (P) = n \geq \mathcal{N}_0 (a)$, then $m \leq \frac{a}{4}$.

\subsection{Towards explicit analogues of D\'{e}got's bounds}

In this section we study \textbf{Theorems 5} and \textbf{6} from D\'{e}got's paper. From \textbf{Theorem 5}, we study closely the quantities that go into the definition of $N_1$. We will see that the bound $N_1$ as originally defined depends on the real part of the mean of the zeroes of the polynomial $P(z)$. This section deals with how to circumvent this dependence.\\

\noindent The end result is that we come up with the formulas $\mathcal{N}_1(a)$ and $\mathcal{N}_2(a)$, the explicit and continuous analogues of D\'{e}got's $N_1$ and $N_2$ respectively. These new quantities have the following advantages over D\'{e}got's:

\begin{itemize}
\item they are explicitly given,
\item they are continuous in $a$,
\item they depend only on $a$.
\end{itemize}

We begin with \textbf{Theorem 5} from \cite{Degot}.

\subsubsection*{Towards $\mathcal{N}_1(a)$}

\begin{lem}{(\cite{Degot}, Theorem 5):}\label{l4} Suppose $P(z)$ contradicts Sendov's conjecture at $a$. Let $q = \frac{\frac{a}{2}-m}{1 + \frac{a}{2}}$ and let $N_1$ be the smallest integer such that

\begin{equation}\label{d32}
\left(\frac{1+\frac{a}{2}}{1+a} \right)^q \leq \left(\frac{1-\sqrt{1-\frac{a^2}{4}}}{an} \right)^{\frac{1}{n-1}} \text{for all} \ n\geq N_1.
\end{equation}

Then, if $n \geq N_1$,

\begin{equation*}
|P'(a)| \leq \frac{16n}{a^2} \ \text{and} \ |P(0)| \geq \frac{a^2}{16}.
\end{equation*}
\end{lem}

\noindent With the help of the quantity $\mathcal{N}_0(a)$ obtained in the previous section, we do this in the following steps:

\begin{itemize}
\item First, we note that by construction, $\mathcal{N}_0(a)$ gives us a high enough degree bound such that any polynomial with degree $n \geq \mathcal{N}_0(a)$ has $m \leq 0.25a$.
\item The quantity $\left( \frac{1+\frac{a}{2}}{1+a} \right) \in (0, 1)$, hence for any $0<q_1 <q_2$:

\begin{equation*}
\left( \frac{1+\frac{a}{2}}{1+a} \right)^{q_1} \geq \left( \frac{1+\frac{a}{2}}{1+a} \right)^{q_2}.
\end{equation*}
\item Therefore this implies that whenever $m \leq \frac{a}{4}$, then

\begin{equation*}
\left(\frac{1+\frac{a}{2}}{1+a} \right)^q \leq \left(\frac{1+\frac{a}{2}}{1+a} \right)^{\frac{\frac{a}{4}}{1+\frac{a}{2}}}.
\end{equation*}

Thus, by Inequality (\ref{d32}), if we have that

\begin{equation*}
\left(\frac{1+\frac{a}{2}}{1+a} \right)^{\frac{\frac{a}{4}}{1+\frac{a}{2}}} \leq \left(\frac{1-\sqrt{1-\frac{a^2}{4}}}{an} \right)^{\frac{1}{n-1}},
\end{equation*}

it would then follow that:

\begin{equation*}
\left(\frac{1+\frac{a}{2}}{1+a} \right)^q \leq \left(\frac{1+\frac{a}{2}}{1+a} \right)^{\frac{\frac{a}{4}}{1+\frac{a}{2}}} \leq \left(\frac{1-\sqrt{1-\frac{a^2}{4}}}{an} \right)^{\frac{1}{n-1}}
\end{equation*}

whenever $m \leq \frac{a}{4}$.

\end{itemize}

\noindent Hence, the version of $N_1(a)$ obtained by replacing $m$ with $\frac{a}{4}$ (and hence $q = \frac{\frac{a}{4}}{1+\frac{a}{2}}$) works for all $m \leq \frac{a}{4}$. With this in mind, we replace the quantity $q(a,m)$ with the new quantity $q'(a) := \frac{\frac{a}{4}}{1+\frac{a}{2}}$ which only depends on $a \in (0,1)$. 

\begin{pro}\label{pro1}  
Let $\mathcal{N}_1(a) = \max \left \{ 9 \left(  \frac{4+2a}{a} \right)^2 , \mathcal{N}_0(a) \right \}$, then for all \newline $n \geq \mathcal{N}_1(a)$, Inequality (\ref{d32}) (with $q'$ in the place of $q$) holds.
\end{pro}

\begin{proof}
We shall construct the function $\mathcal{N}_1(a)$ explicitly:\\

\noindent Replacing $q$ by $q'$ in Inequality (\ref{d32}) and then taking $\log$ on both sides, we get that the new inequality holds if and only if:

\begin{equation*}
q' \log \left( \frac{1+\frac{a}{2}}{1+a} \right) \leq \frac{1}{n-1}\log \left( \frac{1-\sqrt{1-\frac{a^2}{4}}}{an} \right)
\end{equation*}

%

This is true if and only if:

\begin{equation}\label{d4}
n \geq 1 + \frac{\log \left( 1-\sqrt{1-\frac{a^2}{4}}  \right) - \log(a)}{q' \log \left( \frac{1+\frac{a}{2}}{1+a} \right)} - \frac{\log (n)}{q' \log \left( \frac{1+\frac{a}{2}}{1+a} \right)}.
\end{equation}

More succinctly, we write:

\begin{equation}\label{d5}
n \geq 1 + m_1(a) + m_2(a) \cdot \log(n),
\end{equation}

where:

\begin{itemize}
\item $m_1(a) = \frac{\log \left( 1-\sqrt{1-\frac{a^2}{4}}  \right) - \log(a)}{q' \log \left( \frac{1+\frac{a}{2}}{1+a} \right)}$, and
\item $m_2(a) = \frac{-1}{q' \log \left( \frac{1+\frac{a}{2}}{1+a} \right)}$.
\end{itemize}


\noindent Both $m_1(a)$ and $m_2(a)$ are less than $\frac{1}{q'} = \frac{4+2a}{a}$. We let $n_1 = n_1(a) = \frac{4+2a}{a}$.\\

\noindent Proceeding, we have that Inequality (\ref{d32}) will still hold if:

\begin{equation*}
n \geq 1 + n_1 + 2\sqrt{n} n_1 \geq 1 + m_1(a) + m_2(a) \log(n).
\end{equation*}

\noindent And therefore, upon completing the square in the first inequality, this is true when:

\begin{equation*}
n \geq \left[ n_1 + (1 + n_1 + n_1^2 )^\frac{1}{2} \right]^2
\end{equation*}

\noindent Simplifying further, this will be true for any $n \geq 9n_1^2$.\\

We then let:

\begin{equation}\label{d6}
N'_1(a) = 9 [n_1(a)]^2 = 9 \left(  \frac{4+2a}{a} \right)^2.
\end{equation}

\noindent Finally, letting

\begin{equation*}
\mathcal{N}_1(a) = \max \{N_1(a), \mathcal{N}_0(a)  \}
\end{equation*}

\noindent completes the proof.

%

\end{proof} 

\subsubsection*{Towards $\mathcal{N}_2(a,c)$}

Having obtained the explicit formula $\mathcal{N}_1(a)$, we now turn our attention to D\'{e}got's \textbf{Theorem 6}, wherein conditions to be satisfied by the second bound $N_2$ were stipulated. The statement given below stipulates such a condition. We state our version, the only difference from his being that we replaced the appearance of $q$ with $q'$.\\

\begin{defn} Let $c \in (0,a)$. For $x \in (0,1)$ set:

\begin{equation*}
D(x) := \max \left\{ \left(\frac{1}{1+a} \right)^{x} ; \left(\frac{1+c}{1+a}\right)^{x} \left(\sqrt{1+c^2 -ac} \right)^{1-x} \right\}.
\end{equation*}

\end{defn}

It is easy to see that $D(x) <1$ for all $x \in (0,1)$.\\

Proceeding, define $N'_2$ to be the smallest integer such that

\begin{equation}\label{d7}
{D(q')}^{n-1} \leq \frac{a}{16n} \ \text{for all} \ n\geq N'_2.
\end{equation}


\begin{rem}
The role of the quantity $N'_2$ (and $N'_1$) will become apparent when bounding the quantity $|P(c)|$ as mentioned in the introduction. We shall consider this in the next section. In the meantime, we bring the reader's attention to the following:
\end{rem}

\begin{pro} 
Let $\mathcal{N}_2(a,c) = \max \left \{ 9 \left(   \frac{\log \left( \frac{a}{16} \right)}{\log \left(  \frac{c}{1+a} \right) }  \right)^2 , \mathcal{N}_0(a) \right \}$, then for all $n \geq \mathcal{N}_2(a,c)$, Inequality (\ref{d7}) holds.
\end{pro}

\noindent The proof of the above proposition follows the same technique used in the proof of Proposition \ref{pro1}. One starts from Inequality (\ref{d7}) and argues as in Proposition \ref{pro1} to get an inequality of the form $n \geq N'_2(a,c)$. Using the fact that $D(x) \geq \frac{c}{1+a}$ for all $x \in (0,1)$, we can, as in the proof of Proposition \ref{pro1}, obtain simpler estimates of otherwise technical terms. Thus, whenever $n \geq \mathcal{N}_2(a,c) := \max \{ N'_2(a,c), \mathcal{N}_0(a)  \}$, then Inequality (\ref{d7}) holds, with $N_2$ replaced by $\mathcal{N}_2(a,c)$.\\




\section{Bounds on the size of $|P(c)|$}

\subsection{The upper bound of $|P(c)|$}

In this section, we now put to use the bounds $\mathcal{N}_1(a)$ and $\mathcal{N}_2(a,c)$ to obtain bounds on the size of $|P(c)|$. The first result is the analogue of D\'{e}got's {\bf Theorem 6}, which gives the upper bound on $|P(c)|$.

\begin{theorem}\label{T1} Suppose $P(z)$ contradicts Sendov's conjecture at $a \in (0,1)$ and let $c\in (0,a)$. If $\deg (P) = n \geq \max \{\mathcal{N}_1(a), \mathcal{N}_2(a,c)  \}$, then,

\begin{equation*}
|P(c)| \leq 1 + a.
\end{equation*}
\end{theorem}

\noindent The proof of the above theorem is essentially a modification of D\'{e}got's proof of his {\bf Theorem 6}, the only changes being the replacement of the quantities $q$ with $q'$, and $N_1, N_2$ with $\mathcal{N}_1(a)$ and $\mathcal{N}_2(a,c)$ respectively.

\subsection{Towards the lower bound of $|P(c)|$}

\noindent In this section we look at D\'{e}got's \textbf{Theorem 7}. The goal is to obtain constants $C>0$ and $K>1$ such that for large enough degree $n$, the value of $P(c)$ satisfies $|P(c)| \geq C \cdot K^n$. To this end D\'{e}got defined the following new parameters:

\begin{equation*}
p = \frac{\frac{a}{2}-m}{1-\frac{a}{2}}, \ r = \frac{c(a-c)}{2(1-c^2)}, \ \alpha = \frac{\log (\frac{a}{16})}{\log \left( \frac{c+r}{1+cr}  \right)}
\end{equation*}

and

\begin{equation}\label{ek}
K = \min \left \{ (1+c-ac)^p \sqrt{1+c^2-ac}^{1-p};  (1+c)^q \sqrt{1+c^2-ac}^{1-q} \right \},
\end{equation}
 \noindent where $q := \frac{\frac{a}{2}-m}{1+\frac{a}{2}}$ and $m$ is as defined in Definition \ref{v1}. We give the statement of the theorem below, bearing in mind the definition of $N_1$ from Lemma \ref{l4}.

\begin{lem}{(\cite{Degot}, Theorem 7):}\label{l6} For the previously defined parameters, if the degree $n$ of $P(z)$ is such that $n \geq N_1$, then:

\begin{equation*}
|P(c)|\geq \frac{(1-c)(a-c)}{1-ac} r^\alpha K^{n-1}.
\end{equation*}
\end{lem}

\noindent Before proceeding, we would like to bring the reader's attention to two observations:\\

\noindent \textbf{Observation 1}: For $K$ as defined in Equation \ref{ek}, one can always find $c$ sufficiently close to $a$ such that $K>1$. That is:\\

As $c \rightarrow a$,

\begin{equation*}
\begin{split}
&  (1+c-ac)^p \sqrt{1+c^2-ac}^{1-p} \rightarrow (1+a(1-a))^p >1\\ 
\text{and similarly}\\
&  (1+c)^q \sqrt{1+c^2-ac}^{1-q} \rightarrow (1+a)^q >1.
\end{split}
\end{equation*}

\noindent This observation was enough for D\'{e}got's results, however, recall that we ultimately want a degree bound that depends only on $a$. Thus we would like to obtain an explicit formula $c = c(a)$ which will always yield a $c$ (in terms of $a$) close enough to $a$ such that $K>1$. We will in fact also explicitly bound $K$ from below in terms of $a$. We introduce the quantity $p'(a)  = \frac{\frac{a}{4}}{1-\frac{a}{2}} = p'$ to take the place of $p$ in order to avoid the dependence on $m$. First, for ease of notation, from Equation (\ref{ek}) let:

\begin{equation*}
K_1(a,c,p) = (1+c-ac)^{p} \sqrt{1+c^2-ac}^{1-p}
\end{equation*}

\noindent and 

\begin{equation*}
K_2(a,c,q) = (1+c)^{q} \sqrt{1+c^2-ac}^{1-q}.
\end{equation*} 
 
\noindent Our version of $K$ then becomes:

\begin{equation*}
K'  := \min \left\{ K_1(a,c,p'); K_2(a,c,q')  \right\}
\end{equation*}

\noindent One can verify that if $p_1 \geq p_2>0$ and $q_1 \geq q_2 >0$, then:

\begin{equation*}
K_1(a,c,p_1) \geq K_1(a,c,p_2)  \  \text{and} \ \  K_2(a,c,q_1) \geq K_2(a,c,q_2).
\end{equation*}


\noindent With the above discussion in mind, we see that the conclusion of Lemma \ref{l6} still holds with $K'$ in place of $K$ whenever $n \geq \mathcal{N}_1(a)$. This will become more clear in the discussion leading towards our Theorem \ref{T2}, which, \textit{mutatis mutandis}, is a restatement of Lemma \ref{l6}.\\

\noindent We may now proceed and study how one can obtain an explicit formula $c(a)$ and the lower bound for $K'$. In preparation for the result, we need to first recall the following logarithmic inequalities:

\begin{lem}{(Useful $\log$ inequalities):}\label{ll}

\begin{itemize}
\item $\log (1+x) \geq \frac{x}{2}$ for $x \in [0,1]$,
\item $\frac{x}{x+1} \leq \log(1+x) \leq x $ for $x > -1$.
\end{itemize}

\end{lem}

\noindent We proceed to define the quantity $\mu_2(a)$ as follows:

\begin{equation*}
\mu_2(a) = \left[ \left( \frac{1}{2a} \left( \frac{2}{q'} -2 \right) -\frac{1}{2} \right)^2 - \left( \frac{1}{a^2} - \frac{1}{a} \left( \frac{2}{q'} - 2 \right) \right) \right]^\frac{1}{2} + \left[ \frac{1}{2} - \frac{1}{2a} \left( \frac{2}{q'} - 2 \right)   \right],
\end{equation*}

\noindent and note that this expresses the only positive root of the quadratic equation:

\begin{equation}\label{tm1}
x^2 + \left[ \left( \frac{1}{a} \left( \frac{2}{q'} -2 \right) -1 \right)   \right]x + \left( \frac{1}{a^2} - \frac{1}{a} \left( \frac{2}{q'} - 2 \right) \right) = 0.
\end{equation}

\noindent For $a \in (0,1)$, the formula $\mu_2(a)$ simplifies to:

\begin{equation}\label{argg}
\mu_2(a) = \left( \frac{a^4 +4a^3 +16a^2 +32a +64}{4a^4}  \right)^\frac{1}{2} + \frac{a^2-2a-8}{2a^2}.
\end{equation}

\noindent \textbf{Claim}: $0 < \mu_2(a) <1$.

\begin{proof}
Recalling that $q'(a) = \frac{\frac{a}{4}}{1+\frac{a}{2}} = \frac{a}{4+2a}$ , the quadratic Equation (\ref{tm1}) can be written as:

\begin{equation*}
f(x) = x^2 + \beta (a)x + \rho (a) = 0,
\end{equation*}
where:

\begin{equation*}
\beta (a) = \frac{8+2a-a^2}{a^2} \ \ \text{and} \ \ \rho (a) = \frac{-7-2a}{a^2} <0.
\end{equation*}

\noindent One notes that $f(0) = \rho (a) <0$ and $f(1) = \frac{1}{a^2} >0$. Since the constant term of $f(x)$ is negative whilst the leading coefficient is positive, this implies that the other root is negative. The claim follows.

\end{proof}

%

\noindent We may now proceed to state and prove the proposition.

\begin{pro}\label{pip1} Let $\gamma (a) = 0.1a + 0.9$. Then $K_2(a, a\gamma (a), q') > 1$ for all $a \in (0,1)$.
\end{pro}

\begin{proof}
For $a \in (0,1)$ let $\gamma$ be arbitrary and satisfy:

\begin{equation*}
1> \gamma> \left[ \left( \frac{1}{2a} \left( \frac{2}{q'} -2 \right) -\frac{1}{2} \right)^2 - \left( \frac{1}{a^2} - \frac{1}{a} \left( \frac{2}{q'} - 2 \right) \right) \right]^\frac{1}{2} + \left[ \frac{1}{2} - \frac{1}{2a} \left( \frac{2}{q'} - 2 \right)   \right] > 0
\end{equation*}

\noindent Focusing on the middle inequality, bearing in mind that $\mu_2(a)$ is a root of the Equation \ref{tm1}, reversing the ``completion of the square'' with respect to $\gamma$ , yields:

\begin{equation*}
\left[ \gamma+ \left( \frac{1}{2a} \left( \frac{2}{q'} -2 \right) -\frac{1}{2} \right)   \right]^2 > \left( \frac{1}{2a} \left( \frac{2}{q'} -2 \right) -\frac{1}{2} \right)^2 - \left( \frac{1}{a^2} - \frac{1}{a} \left( \frac{2}{q'} - 2 \right) \right).
\end{equation*}

\noindent Continuing to simplify, we eventually arrive at:

\begin{equation*}
q'(1+\gamma^2 a^2-\gamma a^2) + 2(1-q')(a\gamma-a) > 0.
\end{equation*}

\noindent And ultimately:

\begin{equation*}
q' + (1-q') \left( \frac{a\gamma-a}{1+a^2\gamma^2-a^2\gamma}  \right) > \frac{q'}{2}
\end{equation*}

Whence:

\begin{equation}\label{fftt}
q'\frac{(a\gamma)}{2} + \frac{(1-q')}{2} \left( \frac{\gamma^2a^2-\gamma a^2}{1+\gamma^2a^2-\gamma a^2}  \right) > \frac{aq'\gamma}{4}.
\end{equation}

\noindent By Lemma \ref{ll}, we have that:

\begin{itemize}
\item $\log (1+a\gamma) \geq \frac{a\gamma}{2}$, and
\item $\log (1+\gamma^2a^2-\gamma a^2) \geq \frac{\gamma^2a^2-\gamma a^2}{1+\gamma^2a^2-\gamma a^2} $
\end{itemize}

\noindent Using the above inequalities and Equation (\ref{fftt}), we deduce that:

\begin{equation}\label{nana1}
\log (K_2(a, a\gamma, q')) = q'\log(1+a\gamma) + \left( \frac{1-q'}{2} \right) \log(1+\gamma^2a^2-\gamma a^2) > \frac{aq'\gamma}{4}.
\end{equation}

\noindent Hence,

\begin{equation*}
K_2(a,a\gamma, q') = (1+a\gamma)^{q'} (1+\gamma^2a^2-\gamma a^2)^\frac{1-q'}{2} > e^\frac{aq'\gamma}{4} > 1.
\end{equation*}

\noindent Finally, one notes that the function $\gamma$ just has to satisfy $1 > \gamma (a) > \mu_2(a) $ for all $a \in (0,1)$.\\

\noindent The Taylor expansion of $\mu_2(a)$ around $a = 0$ is given by:

\begin{equation*}
\frac{7}{8} + \frac{a}{32} + \frac{3a^2}{512} - \frac{9a^3}{2048} + \frac{7a^4}{16384} + \mathcal{O}(a^5).
\end{equation*}

\noindent In particular,

\begin{equation*}
\lim_{a \rightarrow 0} \mu_2(a) = 0.875.
\end{equation*}

\noindent On the other hand, $\mu_2(1) = \frac{3(\sqrt{13} - 3)}{2} \approx 0.908$.\\

\noindent For ease of notation, from Equation (\ref{argg}) let

\begin{equation*}
t(a) = a^4+4a^3+16a^2+32a+64.
\end{equation*}

\noindent Then:

\begin{equation*}
\mu'_2(a) = \frac{(-a^3-8a^2+24a-64)}{a^3t(a)^\frac{1}{2}} + \frac{(a+8)}{a^3}.
\end{equation*}

\noindent First, a computation via Mathematica shows that

\begin{equation*}
\mu'_2(0.5) \approx 0.034094 > 0.
\end{equation*}

\noindent Secondly, $\mu'_2(a) = 0$ implies, after simplifying, that:

\begin{equation*}
(a+8)^2t(a) = (a^3+8a^2-24a+64)^2.
\end{equation*}

\noindent The real roots of the above equation, computed via Mathematica are \newline $\{ -23.292, -9.7009,0 \}$. Hence $\mu'_2(a) > 0$ on $(0,1)$.\\

\noindent We have thus shown that $\mu_2(a)$ is increasing on $(0,1)$. Furthermore, the function $\gamma (a) := 0.1a + 0.9$ dominates $\mu_2(a)$ for $a \in (0,1)$ as desired.

\end{proof}

\begin{rem}\label{remm}
One notes from Equation (\ref{nana1}) that $\gamma$ was constructed such that $\log (K_2(a,a\gamma, q')) > \frac{a\gamma q'}{4}$.
\end{rem}

\noindent Following the same technique as used to prove Proposition \ref{pip1} above, one can also prove the following:

\begin{pro}\label{pip2} Let $\gamma (a) = 0.1a + 0.9$. Then $K_1(a, a \gamma (a), p') >1$ for all $a \in (0,1)$.
\end{pro}

\begin{proof}

\noindent Since a considerable part of the proof strategy is very much like the that used in the proof of the previous proposition, we omit most of the details and only highlight the relevant parts of the argument.\\

\noindent Let $\mu_1(a)$ be:
 
\begin{equation*}
\mu_1(a) = \frac{(-a^3+a^2+6a-8) + (a^6-2a^5+9a^4-20a^3+48a^2-96a+64)^\frac{1}{2}}{2a^2(1-a)}
\end{equation*}

\noindent This expresses the (only) positive root of a quadratic equation constructed such that $0< \mu_1(a) < 1$ for all $a \in (0,1)$, and for some function $\rho (a)$ if,

\begin{equation*}
0 < \mu_1(a) \leq \rho(a) <1,
\end{equation*}

\noindent then:

\begin{equation*}
\log (K_1 (a, a\rho(a),p')) > \frac{a(1-a)p' \rho(a)}{4} > 0. \ \text{(Compare with Remark \ref{remm} )}.
\end{equation*}

\noindent The main objective is to show that $\gamma(a)$ as defined in the previous proposition suffices for the role of $\rho(a)$ as described above. That is, $\gamma (a) \geq \mu_1(a)$ for all $a \in (0,1)$.\\

\noindent First, we note that by construction:

\begin{equation*}
\sup_{a \in (0,1)} \mu_1(a) \leq 1.
\end{equation*}

\noindent The Taylor expansion of $\mu_1(a)$ around $a = 0$ is given by:

\begin{equation*}
\frac{7}{8} + \frac{a}{32} + \frac{19a^2}{512} + \frac{23a^3}{2048} + \frac{127a^4}{16384} + \mathcal{O}(a^5).
\end{equation*}

\noindent In particular,

\begin{equation*}
\lim_{a \rightarrow 0} \mu_1(a) = \frac{7}{8}.
\end{equation*}


%
%
%
%
%
%
%
%
%
%
%
%

\noindent \textbf{Claim:} The function $\mu_1(a)$ is convex on $(0,1)$.\\

\noindent We have to show that $\mu''_1(a)>0$ on $(0,1)$.\\

\noindent For ease of notation, let 

\begin{equation*}
g(a) = a^6-2a^5+9a^4-20a^3+48a^2-96a+64
\end{equation*}

\noindent Then:

\begin{equation*}
\mu'_1(a) = \frac{4a^5-15a^4+53a^3-144a^2+168a-64}{a^3(1-a)^2 g(a)^\frac{1}{2}} + \frac{6a^2-15a+8}{a^3(1-a)^2}.
\end{equation*}

\noindent It follows that:

\begin{equation*}
\mu''_1(a) = \frac{2}{a^4(1-a)^3 g(a)^\frac{3}{2}} (h(a) + (9a^3-33a^2+35a-12)g(a)^\frac{3}{2}),
\end{equation*}

\noindent where

\begin{equation*}
\begin{split}
h(a) = & \ 6a^{12} -42a^{11} +238a^{10}-1035a^9 +3273a^8 - 8639a^7 +19791a^6\\ 
&-39552a^5+66928a^4-84672a^3+69312a^2-31744a+6144.
\end{split}
\end{equation*}

\noindent Then, $\mu_1''(a) = 0$ implies that

\begin{equation*}
h(a) = - (9a^3-33a^2+35a-12)g(a)^\frac{3}{2},
\end{equation*}

\noindent and hence

\begin{equation}\label{nn3}
h(a)^2 - (9a^3-33a^2+35a-12)^2g(a)^3 = 0.
\end{equation}

\noindent The real roots of the polynomial (\ref{nn3}) are $\{-1.9111, 0, 1, 1.1878, 2.1974 \}$.\\

\noindent Finally, 

\begin{equation*}
\mu''_1(0.5) \approx 0.176399 > 0.
\end{equation*}

\noindent We have thus shown that $\mu_1(a)$ is convex on $(0,1)$. For all $a \in (0,1)$, $\mu_1(a)$ is dominated by the straight line passing through the points $(0,0.9)$ and $(1,1)$.\\

Hence $\gamma(a) > \mu_1(a)$ for $a \in (0,1)$ as desired.

\end{proof}

\noindent \textbf{Observation 2}: It is desirable to take stock of the preceding discussion at this moment. The reader is reminded that the lower bound of $n$ required to obtain the conclusion of D\'{e}got's \textbf{Theorem 7} is the previously defined $N_1$ from his \textbf{Theorem 5} (in our case { Lemma \ref{l4}}). We have already obtained the explicit analogue of this bound in the form of $\mathcal{N}_1(a)$. Hence, as it stands, we have all the necessary ingredients to obtain the conclusion of D\'{e}got's \textbf{Theorem 7}.\\

\noindent However, since our ultimate goal is to obtain an explicit $\mathcal{N}(a)$ independent of all the other implicit parameters, it is worthwhile to remark on the new parameters that were introduced in preparation for Lemma \ref{l6}.

\begin{itemize}
\item $p$ is defined as $p = \frac{\frac{a}{2}-m}{1-\frac{a}{2}} = p(a,m)$. The dependence on $m$ is avoided by the same argument that led to the introduction of $q'(a)$. We simply define the alternative quantity $p' = \frac{\frac{a}{4}}{1-\frac{a}{2}} = p'(a)$ and invoke the quantity $\mathcal{N}_0(a)$ to ensure a high enough degree bound such that the results work.
\item The parameter $r$ is defined in terms of $a \in (0,1)$ and $c \in (0,a)$ as $r = \frac{c(a-c)}{2(1-c^2)}$. The role of the function $\gamma$ as defined in Proposition \ref{pip1} comes into play here. We will define $c$ to be $ a \gamma (a)$, thus obtaining $r = r(a)$, a function depending only on $a \in (0,1)$.
\item Similar reasoning as above applies to the quantity $\alpha = \frac{\log (\frac{a}{16})}{\log \left( \frac{c+r}{1+cr}  \right)}$.
\end{itemize}

\noindent That being said, we arrive at our version of  D\'{e}got's \textbf{Theorem 7} which depends only on $a \in (0,1)$. We restate the conclusion here for the sake of continuity:\\

\begin{theorem}\label{T2} Suppose $P(z)$ contradicts Sendov's conjecture at $a \in (0,1)$. Let $c = a \gamma (a)$. If $\deg (P(z)) = n \geq \mathcal{N}_1(a)$, then:

\begin{equation*}
|P(c)|\geq \frac{(1-c)(a-c)}{1-ac} r^\alpha {K'}^{n-1}.
\end{equation*}
\end{theorem}

\noindent Before proceeding, let us take yet another closer look at these parameters. This analysis will prove useful and simplify notation in the result that follows thereafter.

\begin{itemize}
\item the quantity $r$ is defined as $r = \frac{c(a-c)}{2(1-c^2)}$. It can be shown that $0<r<1$, hence $\log(r) < 0$. Also, since $1-c^2 < 1$, in later analysis we can replace $r$ with $r' := \frac{c(a-c)}{2}$. Theorem \ref{T2} will still be true for $r'$ since $0 < r' < r <1$, and, as we will show below, $\alpha > 0$.



\item $\log \left( \frac{1+a}{a-c} \right) > 0$ and always defined. 

\item $\log \left( \frac{1-ac}{1-c} \right) > 0$ and always defined. 

\item Also, $0 < \frac{c+r}{1+cr} <1$, hence:

%

\begin{equation*}
\alpha := \alpha (r) = \frac{\log (\frac{a}{16})}{\log \left( \frac{c+r}{1+cr}  \right)} > 0.
\end{equation*}

\noindent Furthermore, $\frac{c+r}{1+cr} > \frac{c+r'}{1+cr'}$, and therefore:

\begin{equation*}
\alpha' := \alpha (r') > \alpha (r).
\end{equation*}

\item Finally, we have shown that we can express $c$ explicitly in terms of $a$. Furthermore, this $c$ is sufficiently close to $a$ such that $K'>1$. Hence $\log(K') >0$.

\end{itemize}

\noindent We define the function $N_3(a,c)$ by:

\begin{equation*}
N_3(a,c) = \frac{\log \left( \frac{1+a}{a-c} \right) + \log \left( \frac{1-ac}{1-c} \right) - \alpha \log(r)}{\log (K')} + 1.
\end{equation*}

\noindent All the above analysis culminates in the following definition of the final degree bound, which we denote by $\mathcal{N}_3(a)$ as follows:

\begin{equation*}
\mathcal{N}_3(a) = \max \{ \mathcal{N}_0(a), N_3(a,a \gamma (a)) \}.
\end{equation*}

\section{Improvement of D\'{e}got's Theorem 8}

\noindent We begin this section with our analogue of D\'{e}got's main result.

\begin{theorem}\label{T3} Let $P(z) = (z-a)\prod_{j=1}^{n-1}(z-z_j)$, with $a \in (0,1)$, $|z_j| \leq 1$ for all $j = 1, \ldots , n-1$, where $n \geq 2$. If:

\begin{equation*}
\deg P(z) = n > \mathcal{N}'(a) := \max \{ \mathcal{N}_1(a), \mathcal{N}_2(a, a \gamma (a)), \mathcal{N}_3(a) \},
\end{equation*}

then $P'(z)$ has a zero in the disk $|z-a|\leq 1$.
\end{theorem}

\begin{proof}
We follow D\'{e}got's approach:\\

\noindent Let $c = a \gamma (a)$ and suppose to the contrary, that $P'(w) \neq 0$ for all $w \in |z-a|\leq 1$. From Theorems \ref{T1} and \ref{T2}, we get:

%
%
%

\begin{equation*}
1+a \geq |P(c)| \geq \frac{(1-c)(a-c)}{1-ac} r^\alpha {K'}^{n-1}.
\end{equation*}

\noindent This implies that:

\begin{equation*}
\frac{(1-c)(a-c)}{(1-ac)(1+a)} r^\alpha {K'}^{n-1} \leq 1.
\end{equation*}

%

\noindent Therefore,

\begin{equation*}
(n-1) \log (K') \leq \log \left( \frac{1-ac}{1-c} \right) + \log \left( \frac{1+a}{a-c} \right) - \alpha \log (r).
\end{equation*}

\noindent Hence,

\begin{equation*}
n \leq \frac{\log \left( \frac{1+a}{a-c} \right) + \log \left( \frac{1-ac}{1-c} \right) - \alpha  \log(r)}{\log (K')} + 1 \leq  \mathcal{N}_3(a).
\end{equation*}

\noindent This contradicts the assumption on the degree of $P(z)$.\\

\noindent Hence $P'(w) = 0$ for some $w \in |z-a|\leq 1$.

\end{proof}

\subsection{The explicit function $\mathcal{N}(a)$}

\noindent Thus far, we have all the ingredients that go into constructing the function $\mathcal{N}(a)$. For the sake of completeness we bring them together and express them under one formula. First, for the reader's convenience, we recall the following formulae:\\

\begin{itemize}
\item $\mathcal{N}_0(a) := \frac{32 \log \left( \frac{40}{a^2} \right)}{a^2}$ and hence $\mathcal{N}_0(a) \leq \frac{32 \left( \frac{40}{a^2} \right) }{a^2}  = \frac{1280}{a^4}$,
\item $\mathcal{N}_1(a) := \max \left \{  9 \left(  \frac{4+2a}{a} \right)^2,  \mathcal{N}_0(a)  \right\} \leq  \max \left \{ \frac{324}{a^2} ,  \mathcal{N}_0(a)  \right\}$,
\item $\mathcal{N}_2(a) := \max \left \{ 9 \left(   \frac{\log \left( \frac{a}{16} \right)}{\log \left(  \frac{0.1a^2 + 0.9a}{1+a} \right) }  \right)^2 ,  \mathcal{N}_0(a)  \right\} $. On the interval $(0,1)$, we have that $\log \left(  \frac{1+a}{0.1a^2+0.9a} \right) \geq \frac{2a^2}{5}$. Hence $\mathcal{N}_2(a) \leq \max \left\{ \frac{5760}{a^2}, \mathcal{N}_0(a) \right\}$.
\item We therefore have that $\mathcal{N}_0(a), \mathcal{N}_1(a), \mathcal{N}_2(a) \leq \frac{5760}{a^4}$ for all $a \in (0,1)$.
\end{itemize}

\noindent Before proceeding, let us do a further analysis on the function $\mathcal{N}_3(a)$.\\

\noindent Recall that $\mathcal{N}_3(a) := \max \{ \mathcal{N}_0(a), N_3(a,a \gamma (a)) \}$ where:

\begin{equation*}
N_3(a,a \gamma (a)) = N_3(a) = \frac{\log \left( \frac{1+a}{a-a\gamma} \right) + \log \left( \frac{1-a^2\gamma}{1-a\gamma} \right) - \alpha \log(r)}{\log (K')} + 1.
\end{equation*} 

\noindent We would like to replace $N_3(a)$ with a larger estimate. First we note that:

\begin{equation*}
N_3(a) \leq \frac{\log \left( \frac{1+a}{a-a\gamma} \right) + \log \left( \frac{1-a^2\gamma}{1-a\gamma} \right) - \alpha' \log(r')}{\log (K')} + 1
\end{equation*}

\noindent where:

\begin{equation*}
\alpha' = \frac{\log (\frac{a}{16})}{\log \left( \frac{a\gamma+r'}{1+ar' \gamma}  \right)} \ \ \text{and} \ \ r' =\frac{a^2 (1-\gamma)\gamma}{2}.
\end{equation*}

\noindent Since $\log (x) \leq x-1$ for $x > 0$, we have that:

\begin{itemize}
\item $\log \left( \frac{1+a}{a-a\gamma} \right) \leq \left( \frac{1+a}{a-a\gamma} \right) \leq \frac{2}{a - a\gamma}$,
\item $\log \left( \frac{1-a^2\gamma}{1-a\gamma} \right) \leq \left( \frac{1-a^2\gamma}{1-a\gamma} \right) \leq \frac{1}{a-a\gamma}$,
\item $\log \left( \frac{1}{r'}  \right) \leq \frac{1}{r'} = \frac{2}{a^2 (1-\gamma)\gamma} \leq \frac{2}{a^3(1-\gamma)}$ since $a < \gamma$.
\end{itemize}

\noindent One can show that $\frac{1+ar'\gamma}{a\gamma +r'} > \frac{1}{\sqrt{a}}$ for $a \in (0,1)$. Hence:

\begin{equation*}
\alpha' \leq \left. \frac{16}{a} \right/ \frac{1}{2} \log \left( \frac{1}{a} \right) = \frac{32}{a \log \left( \frac{1}{a} \right)}.
\end{equation*}

\noindent Recall that $\gamma = 0.1a + 0.9 $ was constructed such that for $a \in (0,1)$,

\begin{equation*}
\log (K') > \min \left \{ \frac{a^2 \gamma}{4(4+2a)} ;  \frac{a^2 (1-a)\gamma}{4(4-2a)} \right \} = \frac{a^2 (1-a) \gamma}{4(4-2a)} \geq \frac{a^2 (1-a) \gamma}{16} \geq \frac{a^3(1-a)}{16}.
\end{equation*}

\noindent We thus have that:

\begin{align*}
N_3(a) & \leq \left( \frac{3}{a (1-\gamma)} + \frac{2}{a^3(1-\gamma)} \left( \frac{32}{a \log \left( \frac{1}{a} \right)}  \right)    \right) \frac{16}{a^3(1-a)} + 1\\
& = \left(   \frac{3 a^3\log \left( \frac{1}{a} \right) + 64 }{a^4 (1-\gamma)\log \left( \frac{1}{a} \right)}   \right) \frac{16}{a^3(1-a)} + 1.
\end{align*}

\noindent We note that $\lim_{a \rightarrow 0} a^3 \log (\frac{1}{a}) = 0$. Also, the function $a^3 \log (\frac{1}{a})$ attains its global maximum of $\frac{1}{3e}$ at $a = \frac{1}{\sqrt[3]{e}}$. Furthermore, $1-\gamma = 0.1(1-a)$. We conclude that:

\begin{equation*}
N_3(a) \leq \left(  \frac{650}{a^4(1-a) \log (\frac{1}{a}) }  \right) \frac{32}{a^3(1-a)} = \frac{20800}{a^7(1-a)^2 \log (\frac{1}{a})}.
\end{equation*}

\noindent For $a \in (0,1)$, $\log (\frac{1}{a}) \geq (1-a)^2$. Hence:

\begin{equation*}
N_3(a) \leq \frac{20800}{a^7(1-a)^4}.
\end{equation*}

\noindent Let $\mathcal{N}(a)$ be given by:

\begin{equation*}
\mathcal{N}(a) = \frac{20800}{a^7(1-a)^4}.
\end{equation*}

\subsection*{Towards uniformity}


\noindent In \cite{Degot}, D\'{e}got concludes by asking for a degree bound $N \in \mathbb{N}$ which is independent of $a \in (0,1)$, or at least an explicit formula $N(a)$.

\noindent We note that, the formula $\mathcal{N}(a)$ defined above suffices for the latter request.\\


\noindent Furthermore, for any $ 0<\alpha < \beta <1$, the extreme value theorem tells us that $\mathcal{N}(a)$ has a maximum on $[\alpha , \beta]$.\\





\noindent The table below compares values of $N$ D\'{e}got computed for certain values of $a$ with the corresponding rounded up approximate values of $\mathcal{N}(a)$.\\

\begin{tabu} to 0.6\textwidth { | X[c] | X[c] | X[c] | }
 \hline
 $a$ & D\'{e}got's $N$ & $\mathcal{N}(a)$ \\
 \hline
0.1  & 15064  & $3.4 \times 10^{11}$  \\
\hline
0.2 & 3587 & $4 \times 10^9$ \\
 \hline
0.3  & 1654  & $4 \times 10^8$  \\
\hline
0.4 & 1004 & $9.8 \times 10^7$ \\
 \hline
0.5  & 718  & $4.3 \times 10^7$  \\
\hline
0.6 & 563 & $3 \times 10^7$ \\
 \hline
0.7  & 560  & $3.2 \times 10^7$  \\
\hline
0.8 & 616 & $6.2 \times 10^7$ \\
 \hline
0.9  & 1006  & $4.4 \times 10^8$  \\
\hline
\end{tabu}

\begin{rem}
It is evident that the above values produced by $\mathcal{N}(a)$ are several orders of magnitude worse than those obtained by D\'{e}got's procedure. We however have the advantage of an explicit formula that works for all $a \in (0,1)$.
\end{rem}

\subsection*{Related results}

\noindent Whilst writing the current paper, we came across the following result which is similar in nature to what is obtained here. In what follows, given a polynomial $P$ with a zero at $z = a$, let $d(a, P)$ denote the distance from $a$ to the nearest critical point of $P$, and $r(a,P)$ denote the distance from $a$ to the nearest \textit{other} zero of $P$.

\begin{theorem}{(T-S Small, Theorem 1(b), \cite{Small}, p.~217)}\label{tssh}
Let $P$ be a polynomial of degree $n \geq 2$ with all its critical points in $\{ |z| \leq 1  \}$, with $P(a) = 0$. Let $d = d(a,P)$ and $r = r(a, P)$. If $|a| \leq 1$ then either:

\begin{equation*}
d \leq \sqrt{1 - |a|^2} \ \ \text{or} \ \ n \leq 2 + \frac{12|a|^2d^2 (r^2 + 4d^2)}{r^2 (|a|^2 + d^2 -1)^2}.
\end{equation*}

\end{theorem}

\noindent The main difference is that our bound does not depend on $r(a,P)$, a quantity that may change given a different polynomial. Indeed, an explicit bound on $r(a,P)$ defines a family of polynomials. For the sub-family of those polynomials consisting of those which satisfy the hypothesis of Sendov's conjecture, one can via the above theorem compute a degree lower bound $N(a)$ such that if $n \geq N(a)$, then the conjecture is true at $z = a$.\\

\noindent As a corollary of Theorem \ref{tssh}, T-S Small obtains the following rather nice bound:

\begin{theorem}{(T-S Small, Theorem 6.5.7, \cite{Small}, p.~220)}\label{tssh2}
Let $P$ be a polynomial of degree $n$ such that $P(a) = 0$ for $a \in (0,1)$. Suppose further that all the remaining zeroes of $P$ lie {\bf on the unit circle}. Then:

\begin{equation*}
d(a,P) \leq 1  \ \text{for} \ n > 2 + \frac{60 - a^2}{a^2(1-a^2)}.
\end{equation*}

\end{theorem}

\noindent At the expense of sharpness, our result is more general since we do not make an assumption on the distribution of the rest of the other zeroes of $P$.

\subsection*{Concluding remarks}

\noindent In conclusion, we would like to bring the reader's attention to the following points:

\begin{itemize}
\item We would like a definitive result that would bridge the gaps $[0, \alpha)$ and $(\beta , 1]$. We are of the opinion that these gaps rather illustrate the limitation of this current approach, as opposed to the veracity of the conjecture.
\item Our main goal was to find a simple explicit formula $\mathcal{N}(a)$, we thus overestimated many functions by replacing them with simpler formulae. In view of D\'{e}got's computations, the results we obtained here can still be considerably improved.
\end{itemize}


\vspace*{0.3cm}

\begin{center}

\textbf{\large{Acknowledgement}}

\end{center}

\noindent The author would like to thank Gareth Boxall for many insightful comments and suggestions.


\vspace*{0.3cm}

\noindent Department of Mathematical Sciences, \newline Mathematics Division, \newline Stellenbosch University.\\

\noindent \textit{e-mail address}: \ taboka@aims.ac.za

\end{document}